\documentclass[a4paper,11pt]{amsart}
\usepackage[plainpages=false]{hyperref}
\usepackage{amsfonts,latexsym,rawfonts,amsmath,amssymb,amsthm, mathrsfs, lscape}
\usepackage{verbatim}

\usepackage[all]{xy}
\usepackage{graphicx,psfrag}

\usepackage{array, tabularx}

\usepackage{setspace}

\newtheorem{thm}{Theorem}

\newtheorem{cor}{Corollary}[section]
\newtheorem{lem}{Lemma}[section]

\newtheorem{prop}{Proposition}[section]

\newtheorem{conj}[thm]{Conjecture}

\theoremstyle{remark}
\newtheorem{rmk}{Remark}[section]

\theoremstyle{definition}
\newtheorem{defi}{Definition}[section]

\numberwithin{equation}{section}

\def\p{\partial}
\def\R{\mathbb{R}}
\def\C{\mathbb{C}}

\def\l{\lambda}

\def\i{\sqrt{-1}}
	
\def\D{\Delta}

\def\cE{{\mathcal E}}
\def\cF{{\mathcal F}}

\def\cH{{\mathcal H}}

\def\cK{{\mathcal K}}
\def\cL{{\mathcal L}}

\def\cX{{\mathcal X}}

\begin{document}

\title[Kahler-Ricci soliton and $H$-functional]{Kahler-Ricci soliton and H-functional}

\author{Weiyong He}

\address{Department of Mathematics, University of Oregon, Eugene, Oregon, 97403}
\email{whe@uoregon.edu}

\begin{abstract}We consider  K\"ahler-Ricci soliton on a Fano manifold $M$. We introduce an $H$-functional on $M$; we show that its critical point has to be a Kahler-Ricci soliton and the Kahler-Ricci flow can be viewed as its reduced gradient flow.  We then obtain a natural lower bound of $H$-functional in terms of an invariant of holomorphic vector fields on $M$. As an application, we prove that a Kahler-Ricci soliton, if exists, maximizes Perelman's $\mu$-functional. Second we consider a conjecture  proposed by S.K. Donaldson regarding the existence of Kahler metrics with constant scalar curvature in terms of $\cK$-energy; a simple observation is that on Fano manifolds, one can consider Donaldson's conjecture in terms of Ding's $\cF$-functional. We then state geodesic stability conjecture on Fano manifolds in terms of $\cF$-functional.  Similar pictures can be naturally extended  to a Kahler-Ricci soliton and modified $\cF$-functional.  
\end{abstract}

\maketitle

\section{Introduction}

In this paper we consider K\"ahler-Ricci soliton on a Fano manifold $M$.  Ricci solitons were defined by R. Hamilton in the study of Ricci flow \cite{Hamilton82}  and it plays a significant role in the theory of Ricci flow.  A Kahler-Ricci soliton on a Fano manifold is Kahler-Einstein precisely when the Futaki invariant vanishes and hence Kahler-Ricci soliton is a natural generalization of Kahler-Einstein metrics.  In his study of Hamilton's Ricci flow, Perelman \cite{Perelman01} introduced many revolutionary ideas, including the well-known entropy functionals, which lead him to the solution of the Poincare conjecture and ThurstonÕs geometrization conjecture. Ricci flow is the gradient flow of Perelman's $\mu$-functional (modulo diffeomorphisms).  
On Fano manifolds, Kahler-Ricci solitons are critical points of Perelman's $\mu$-functional.  A Kahler-Einstein metric, if exists, maximizes $\mu$-functional. Kahler-Ricci soliton is also expected to maximize $\mu$-functional.  This is actually the case \cite{TZ3},  for example,  if we consider only metrics which are invariant with respect to the imaginary part of extremal vector field studied in Tian-Zhu \cite{Tianzhu02}. 

Our main observation is to consider a functional (this quantity has been studied along the Kahler-Ricci flow in literature) on Fano manifolds.  We show that a Kahler-Ricci soliton is a critical point of this functional. We then consider a lower bound of this functional and  this turns out to be rather straightforward for invariant metrics. In general, we need to consider a natural geometric structure of the space of Kahler potentials, studied by Mabuchi \cite{Mabuchi87}, Semmes \cite{Semmes} and Donaldson \cite{Donaldson99}. The key notion is the geodesic segments and the geodesic rays, in the manner of X.-X. Chen \cite{Chen00, Chen09}.  We then obtain a natural lower bound, called H-invariant which was studied in \cite{TZ3} and this result relies on the key technique fact  of the convexity of Ding's $\cF$-functional \cite{Ding}, established by Berndtsson \cite{Bern06, Bern11}.  As a direct application, we prove that Kahler-Ricci soliton, if exists, maximizes Perelman's $\mu$-functional.

Donaldson \cite{Donaldson99} formulated a conjecture relating existence of constant scalar curvature with the geometric structure of the space of Kahler potentials, in particular with the limit behavior of (derivative of) Mabuchi's $\cK$-energy  (defined in \cite{Mabuchi87}) along geodesic rays. It is very intuitive to understand Donaldson's conjecture in terms of critical points of  the (formally) convex functional, Mabuchi's $\cK$-energy. The constant scalar curvature metric is a critical point of $\cK$-energy (actually minimizer) and $\cK$-energy is (formally) convex in terms of geodesics in the space of Kahler potentials.  Donaldson's conjecture naturally leads to the geodesic stability which was first introduced by Chen \cite{Chen09} in terms of $\rho$-invariant along geodesic rays (see Phong-Sturm \cite{Phongsturm07} also for some related work). Chen \cite{Chen08} made one further step to study the lower bound of $\cK$-energy and partially confirmed Donaldson's conjecture. 

On Fano manifold, it is natural to consider  the geodesic stability in terms of $\cF$-functional. The key advantage of $\cF$-functional is that it is actually convex for geodesic rays with very weak regularity by Berndtsson \cite{Bern11}. This key feature has already been explored by Berndtsson \cite{Bern11} and others (see Berman \cite{Berman, Berman12} for example). For us this removes the main technical obstacle caused by the rather weak regularity of geodesic segments and geodesic rays. Our discussion can naturally be extended to the Kahler-Ricci solitons,  using the modified $\cF$-functional and the modified Futaki invariant, studied by Tian-Zhu \cite{Tianzhu00, Tianzhu02}. We then formulate a version of geodesic stability for Kahler-Ricci solitons on Fano manifolds.  

We organize the paper as follows. In Section 2 we study the $H$-functional and its relation with Kahler-Ricci solitons. In Section 3 we discuss the geodesic stability in terms of $\cF$-functional and modified $\cF$-functional for Kahler-Einstein metrics and Kahler-Ricci solitons.\\

{\bf Acknowledgement:} The author is partially supported by an NSF grant, award No. DMS-1005392.  The author thanks the referee for numerous valuable suggestions and comments, which help to improve the expository of the paper significantly.
The author is very grateful to Song Sun for numerous discussions. The author also thanks Prof. X.-X. Chen for constant support and encouragements. His series work on geodesic stability \cite{Chen09, Chen08} has definite influence on the present work.

\section{Kahler-Ricci soliton and $H$-functional}
\subsection{$H$-functional on Fano manifolds}
Let $(M, [\omega_0])$ be a compact Fano manifold. For any K\"ahler metric $\omega\in [\omega_0]$, the Ricci potential $h$ of $\omega$ is defined to be
\begin{equation}\label{Ricci-potential}
Ric(\omega)-\omega=\sqrt{-1}\p\bar \p h,
\end{equation}
with the normalization condition
\begin{equation}\label{E-hn}
\int_M e^{h}\omega^n=\int_M\omega^n=V, 
\end{equation}
where $V=(2\pi)^n[c_1(M)]^n.$
We define \emph{$H$ functional} as follows, 
\begin{equation}
H(\omega)=\int_M  h e^{ h}  \omega^n. 
\end{equation}
It is worthwhile to point out that $H$-functional can be viewed as the  $H$-entropy in probability theory applied to  the measure $V^{-1}e^h\omega^n$ with respect to the measure $V^{-1}\omega^n$. The celebrated Csisz\'ar-Kullback-Pinsker inequality asserts that, on a complete metric and separate space $\cX$, for any two probability measures $\nu, \mu$, we have, 
\begin{equation}\label{CKP}
\|\nu-\mu\|_{TV}\leq \sqrt{2H(\nu|\mu)},
\end{equation}
where $\|\nu-\mu\|_{TV}$ is the total variation and \[H(\nu|\mu)=\int_\cX \log\frac{d\nu}{d\mu}d\nu\] is the so-called $H$-entropy. We refer the readers to a survey paper \cite{GL} for example, for a nice proof of \eqref{CKP}. Apply \eqref{CKP} with $\mu=V^{-1}\omega^n, \nu=V^{-1}e^h\omega^n$, we get
\begin{equation}\label{CKP1}
\frac{1}{2V}\left(\int_M|1-e^h|\omega^n\right)^2\leq \int_M he^h\omega^n=H(\omega). 
\end{equation}
Hence $H(\omega)$ is a norm-like functional and it is zero precisely when $h=0$, namely when $\omega$ is a Kahler-Einstein metric. 

Recall a Kahler-Ricci soliton satisfies that $\nabla h$ is the real part of a holomorphic vector field $X$; equivalently, the metric satisfies
\[
Ric(g)=g+L_Xg.
\]
When $X=0$ the metric is then a Kahler-Einstein metric with positive scalar curvature. 
In \cite{Tianzhu00, Tianzhu02} Tian-Zhu  proved  the uniqueness of a Kahler-Ricci soliton modulo automorphisms of $M$, extending Bando-Mabuchi's uniqueness theorem \cite{BM} on Kahler-Einstein metrics on Fano manifolds. In particular they proved that $X$ is determined \emph{a priori} by $(M, [\omega_0])$ and it is unique up to automorphism.  
In general we will show that there is a close relation of $H$-functional with a Kahler-Ricci soliton.  
The main result of the paper is the following,

\begin{thm}\label{T-noninvariant}
For any metric $\omega\in [\omega_0]$, there exists a nonnegative numerical invariant $N_X$ of $(M, [\omega_0])$ such that
\begin{equation}\label{E-noninvariant}
H(\omega)\geq N_X,
\end{equation}
where the equality holds if and only of $\omega$ is a Kahler-Ricci soliton in $[\omega_0]$. 
\end{thm}
The invariant $N_X$ in terms of the holomorphic vector field $X$ appears in a recent paper of Tian-Zhang-Zhang-Zhu \cite{TZ3} and the relevant definitions will be recalled below. 

Before we prove Theorem \ref{T-noninvariant}, we shall first explore some interesting properties of $H$-functional and its relation with a Kahler-Ricci soliton. 
We will need the following result due to A. Futaki \cite{Futaki88}.

\begin{prop}[Futaki]\label{Futaki}Let $(M, [\omega_0])$ be a Fano manifold. Suppose $h$ is the normalized Ricci potential of $\omega\in [\omega_0]$, then \[
L_h u=-(\Delta u+\nabla u\nabla h+u)
\]
is a self-adjoint positive operator with respect to $e^h\omega^n$.
In particular, the modified Poincare inequality holds
\begin{equation}\label{E-poincare}
\int_M u (L_hu) e^h\omega^n=\int_M (|\nabla u|^2-u^2)e^h \omega^n\geq 0. 
\end{equation} where  $u$ satisfies the normalized condition
\[
\int_M u e^h \omega^n=0. 
\]
The equality holds if and only if  $L_h u=0$ and  it is equivalent to that $\nabla u$ is a real holomorphic vector field.
\end{prop}
First we have the following, 
\begin{prop}The Euler-Lagrangian equation of $H(\omega)$ is given by
\[
\Delta h+|\nabla h|^2+h=constant.
\]
Equivalently $\nabla h$ is a real holomorphic vector field. Hence a critical point of $H(\omega)$ is a Kahler-Ricci soliton. 
\end{prop}

\begin{proof}
We compute the first variation of $H(\omega)$ directly as follows. 
We write $\omega$ in terms of its Kahler potential, \[\omega=\omega_\phi=\omega_0+\i \p\bar \p \phi.\]
Suppose the variation of $\phi$ is given by $\delta \phi$.
By the definition of Ricci-potential \eqref{Ricci-potential}, we compute,  
\[
\p \bar \p (\delta h)=\p\bar \p (-\Delta \delta \phi-\delta \phi).
\]
Since $H(\omega)$ does not depend on the choice of normalization Kahler potential, we can choose a normalization of Kahler potential such that
\[
\delta h=-\Delta \delta \phi-\delta \phi.
\]
The normalization condition \eqref{E-hn} then implies that $\delta \phi$ satisfies
\[
\int_M \delta \phi e^h \omega^n=0.
\]
We compute
\begin{equation}\label{E-Hvar}
\begin{split}
\delta H(\omega)&=\int_M \left(-(\Delta \delta \phi+\delta\phi)e^h -he^h(\Delta \delta\phi+\delta\phi)+he^h\Delta \delta\phi\right)\omega^n\\
&=-\int_M \delta\phi (\Delta h+|\nabla h|^2+h)e^h\omega^n. 
\end{split}
\end{equation}
The Euler-Lagrangian equation is given by
\begin{equation}\label{E-potential}
\Delta h+|\nabla h|^2+h=constant. 
\end{equation}
Integrating \eqref{E-potential} with respect to $e^h\omega^n$,  the constant above is given exactly by $V^{-1}H$. 
Applying Proposition \ref{Futaki} to $h-V^{-1}H$,  \eqref{E-potential} implies that $\nabla h$ is a real holomorphic vector field and hence a critical point of $H(\omega)$ is a K\"ahler-Ricci soliton. 
\end{proof}
Along the Kahler-Ricci flow
\[
\frac{\p \omega}{\p t}=\omega-Ric(\omega), 
\]
one can compute directly that
\begin{equation}
\frac{\p H}{\p t}=-\int_M\left(|\nabla h|^2-(h-V^{-1}H)^2\right)e^h\omega^n\leq 0. 
\end{equation}
The equality holds exactly when $\omega$ is a K\"ahler-Ricci soliton. 
By \eqref{E-Hvar} the gradient flow of $H(\omega)$ is given by
\[
\frac{\p \phi}{\p t}=\Delta h+|\nabla h|^2+h
\] 
and this is a fourth-order equation. However, since the operator 
\[
L_h u=-(\Delta u+\nabla u\nabla h+u)
\]
is a self-adjoint positive operator with respect to $e^h\omega^n$, we can choose naturally $\delta \phi=-h$ to decrease $H$-functional, with the corresponding flow by
\[
\frac{\p \phi}{\p t}=-h,
\]
which is exactly the K\"ahler-Ricci flow on potential level.  The advantage is that this reduces the flow equation by two orders (compared with the gradient flow). This  formal picture resembles for a general Kahler class that of the Calabi energy, extremal metrics and the Calabi flow \cite{Calabi82};  $H(\omega)$ plays the similar role as the Calabi energy. 

Next we introduce some numerical invariants of $(M, [\omega_0])$ in terms of holomorphic vector fields, in particular we recall the definition of $N_X$. 
Let $Aut_\C(M)$ be the automorphism group of $M$ and $G$ be a maximal compact subgroup.  Suppose further  that $\omega$ is a $G$-invariant metric. Then there is a Lie algebra homomorphism from $Lie(G)$ to the functions on $M$, under Poisson bracket. Let $\xi\in Lie(G)$ and let $\theta_\xi$ be the corresponding Hamiltonian; namely
$d\theta_\xi=\iota_\xi \omega $ with a normalization condition for $\theta_\xi$,
\[
\int_M e^{\theta_\xi} \omega^n=\int_M \omega^n=V. 
\] 
Define the integral
\begin{equation}\label{E-Hinvariant}
H_0(\xi, \omega)=\int_M \theta_\xi e^h \omega^n.
\end{equation}
Tian-Zhang-Zhang-Zhu proved that  (\cite{TZ3} Section 5) that $H_0(\xi, \omega)$ is independent the choice of $\omega$, and hence it defines a numerical invariant on $Lie(G)$. We denote it simply by $H_0(\xi)$ for $\xi \in G$;  moreover, Tian-Zhang-Zhang-Zhu proved that
$H_0$ is a concave function in $Lie(G)$ with a unique maximizer $\xi_0\in Lie(G)$, where $\xi_0$ is  the imaginary part of  the extremal vector field $X$. The invariant $N_X$ is then defined by
\[
N_X=H_0(\xi_0)=\max_{\xi\in Lie(G)} H_0(\xi). 
\]
It is also proved that in \cite{TZ3} that $N_X\geq 0$ and it is zero precisely when $X=0$, or equivalently, the Futaki invariant is zero. 

Suppose we only consider $G$-invariant metrics, then one can easily prove that $H(\omega)\geq N_X$, as a special case of Theorem \ref{T-noninvariant}. Since the proof is straightforward, we include the argument here. 
\begin{prop}\label{P-invariant}Let $\omega$ be a $G$-invariant metric, then
\[
H(\omega)\geq N_X=\max_{\xi\in Lie(G)}H_0(\xi)=H_0(\xi_0). 
\]
\end{prop}

\begin{proof}
Suppose $\omega$ is $G$-invariant and let $\xi\in Lie(G)$. Let $h, \theta_\xi$ be Ricci potential and Hamiltonian of $\xi$ with respect to $\omega$, satisfying the normalization
\[
\int_M e^h\omega^n=\int_M e^{\theta_\xi}\omega^n=V. 
\] 
We only need to show
\[
\int_M (h-\theta_\xi)e^h\omega^n\geq 0
\]
This follows from Proposition \ref{P-convexity} below. 
\end{proof}

The following elementary property will be important for us. 
\begin{prop}\label{P-convexity}Let $f, g$ be two continuous functions, then we have
\begin{equation}\label{E-convexity}
\int_M e^{g}(f-g)\omega^n\leq \int_M e^{f}\omega^n-\int_M e^{g}\omega^n\leq \int_M e^{f}(f-g)\omega^n
\end{equation}
\end{prop}

\begin{proof}
 Suppose $\varphi: \R\rightarrow \R$ is a (smooth) convex function, then for any $f, g: M\rightarrow \R$, we define $F: [0, 1]\rightarrow \R$ by
\[
F(t)=\int_M \varphi(tf+(1-t)g)\omega^n. 
\]
Then $F$ is convex since
\[
F^{''}(t)=\int_M \varphi^{''}(tf+(1-t)g)(f-g)^2\omega^n\geq 0. 
\]
Hence  by convexity,
\[
F^{'}(0)\leq F(1)-F(0)\leq F^{'}(1). 
\]
We have  $F(0)=\int_M \varphi(g) \omega^n, F(1)=\int_M \varphi(f)\omega^n$, $F^{'}(0)=\int_M \varphi^{'}(g)(f-g), F^{'}(1)=\int_M \varphi^{'}(f) (f-g)$. Taking $\varphi(t)=e^t$ we get \eqref{E-convexity}. 
\end{proof}

\subsection{Proof of Theorem \ref{T-noninvariant}}
 X.-X. Chen \cite{Chen09} and S. Donaldson \cite{Donaldson05} proved that the Calabi energy is bounded below by a natural invariant; an extremal metric, if exists, realizes such a lower bound. Chen uses  deep estimates of homogeneous complex Monge-Ampere equations in the space of Kahler potentials \cite{Chen00, Chentian05} and S. Donaldson \cite{Donaldson05} uses finite dimensional approximations (for projective manifolds). In Chen's argument, the geometric structure of the space of Kahler potentials plays an important role. 
We will mimic Chen's approach to prove Theorem \ref{T-noninvariant}. 

We first recall the space of Kahler potentials $\cH$,
\[
\cH=\{\phi\in C^\infty: \omega_\phi=\omega_0+\i \p\bar \p \phi>0\}.
\] 
The Mabuchi metric \cite{Mabuchi86} on $\cH$ is defined as, for $\psi_1, \psi_2\in T_\phi\cH$,
\[
\langle \psi_1, \psi_2\rangle_\phi=\int_M \psi_1\psi_2 \omega_\phi^n. 
\]
For any path $\phi(t)\in \cH$, the geodesic equation is given by
\begin{equation}\label{E-geodesic}
\ddot\phi-|\nabla\dot\phi|^2_{\omega_\phi}=0,
\end{equation}
where we use the complex notation of gradient and Laplacian etc. 
For any interval $I$ in $\R$, denote $U=I\times S^1$. We use $(z, w)$ to denote points on $M\times U$. Then the geodesic equation is equivalent to the homogeneous complex Monge-Ampere equation \cite{Semmes, Donaldson99} (assuming for each $w$, $\phi$ defines a strictly positive Kahler metric), 
\begin{equation}\label{E-hcma}
\Omega_\phi^{n+1}=0,
\end{equation}
where $\Omega_\phi=\pi^{*}\omega_0+\p\bar \p_{w, z} \phi$, $\pi: M\times U\rightarrow M$ is the projection onto $M$ and $\phi$ is regarded as a $S^1$ invariant function on $M\times U$.
A fundamental result of Chen \cite{Chen00} asserts that for $I=[0, 1]$ and $\phi_0, \phi_1\in \cH$, there exists a unique $C^{1, 1}$ solution of \eqref{E-hcma} in the sense that
\begin{equation}\label{E-chen}
\|\phi\|_{1, 1}:=\|\phi\|_{C^1}+\max\{|\p\bar\p_{w, z} \phi|\}\leq C.
\end{equation}
Here $\phi(w, z)$ is regarded as a function on $M\times U$.  We emphasize that Chen's  estimates rely on the fact that the two end points are actually smooth Kahler potentials in $\cH$. 
It is also useful to  consider generalized Kahler potentials. The minimal requirement is that  $\omega_\phi=\omega_0+\i\p\bar \p \phi$ defines a closed positive $(1, 1)$ current. If $\phi\in L^\infty$, then $\omega^n_\phi$ is a well-defined volume form such that \[\int_M \omega^n_\phi=\int_M \omega_0^n=V. \]  In this case we call $\phi$ a bounded Kahler potential and we denote the set of all bounded Kahler potentials as $\cH_{\infty}$.
We also consider generalized $C^{1, 1}$ Kahler potentials defined as
\[\cH_{1, 1}=\{\phi: \omega_\phi\geq 0, \|\phi\|_{C^1}<\infty, 0\leq n+\D\phi<\infty\}.\]
We also define the (weak) $C^{1, 1}$ norm  on $M$ as follows ($\phi$ is a function on $M$), fixing a background metric, 
\[\|\phi\|_{1, 1}^w=\|\phi\|_{C^1}+\max |\D\phi|.\]
If $I=[0, \infty)$ and $\phi(t)$ satisfies \eqref{E-hcma}, then $\phi(t)$ is called a geodesic ray. It is called a bounded geodesic ray if $\phi(t)\in \cH_\infty$ for each $t$ and it is called a $C^{1, 1}$ geodesic ray if $\phi(t)\in \cH_{1, 1}$ for each $t$. Note that we do not specify any condition on $\phi_{tt}$ (or $\phi_{w\bar w}$).

First we state an interesting result proved by Berndtsson \cite{Bern09},
\begin{prop}[Berndtsson]\label{P-norm}
Let $\phi(t), t\in [0, T]$ be the unique geodesic such that the end points $\phi(0), \phi(T)\in \cH$. Suppose $f: \R\rightarrow \R$ is a $C^1$ function such that 
\[
\int_M f(\dot \phi)\omega_\phi^n
\]
is integrable. Then  the above integral is a constant for $t\in [0, T]$.
\end{prop}

Now we are ready to prove Theorem \ref{T-noninvariant}. 
\begin{proof}[Proof of Theorem \ref{T-noninvariant}]To prove \eqref{E-noninvariant}, we only need to show that for any $\omega$ and $\xi\in G$, \begin{equation}\label{E-h1}
H(\omega)=\int_M he^h\omega^n\geq H_0(\xi)
\end{equation}
Fix a background metric $\omega_0$ and we assume $\omega_0$ is a $G$-invariant metric.
For any smooth K\"ahler metric  $\omega$, we write $\omega=\omega_0+\i \p\bar \p \phi$. The Ricci potentials are related by
\[h=h_0-\log\frac{\omega^n}{\omega^n_0}-\phi+constant.
\]
Hence we get,
\begin{equation}\label{E-measure}
e^h\omega^n=\l e^{h_0-\phi} \omega_0^n
\end{equation}
for some positive constant $\l$; by the normalization condition \eqref{E-hn}, we have
\[
\l=V\left(\int_M e^{h_0-\phi}\omega_0^n\right)^{-1}. 
\] 
The relation \eqref{E-measure} is fundamental for us. We will understand  $e^h\omega^n$ as in \eqref{E-measure}  for any closed positive $(1, 1)$ current $\omega\in [\omega_0]$ provided that $\phi$ is only assumed to be bounded.  For any $\xi\in Lie(G)$, we denote $\theta_\xi$ to be its Hamiltonian with respect to $\omega_0$; namely, 
\[
d\theta_\xi=\iota_\xi \omega_0, \int_M e^{\theta_\xi}\omega_0^n=V.
\]
By the definition of $H_0(\xi)$, we then need to show
\[
\int_M he^h\omega^n\geq H_0(\xi)=\int_M \theta_\xi e^{h_0}\omega_0^n. 
\]
Now let  $Y=-J\xi-i\xi$ be the holomorphic vector field and $\sigma_t$ be a one-parameter holomorphisms  generated by $Y$ ($\sigma_0=id$). Denote 
\[
\omega_{\rho(t)}=\sigma_t^{*}\omega_0=\omega_0+\i \p\bar\p \rho(t).\]  We choose $\rho(0)=0$ and $\rho(t), t\in (-\infty, \infty)$ is a smooth geodesic line such that
\begin{equation}\label{E-geodesic1}
\dot \rho=-\sigma_t^* \theta_\xi. 
\end{equation}
We pick up a geodesic segment $\phi(t)$ starting at $\phi$ (the geodesic  will be specified later).  By \eqref{E-convexity}, we know that
\[
\int_M (h-f)e^h\omega^n\geq \int_M e^h\omega^n-\int_M e^f\omega^n. 
\]
By choosing $f=-\dot \phi(0)$, we get (note that $\omega=\omega_{\phi(0)}$)
\begin{equation}\label{E-h2}
H(\omega)\geq -\int_M \dot\phi(0) e^h \omega^n+\int_M e^h\omega^n-\int_M e^{-\dot \phi(0)} \omega^n.
\end{equation}
If  $\omega$ is also $G$-invariant, then as above we can choose $\phi(t)$ (a geodesic line) starting at $\phi$ such that $\dot \phi(t)=-\sigma^{*}_t \tilde \theta_\xi$ (in particular, $\dot \phi(0)=-\tilde \theta_\xi$), where $\tilde \theta_\xi$ is the normalized Hamiltonian of $\xi$ with respect to $\omega$.  Then by  \eqref{E-h2}, we get $H(\omega)\geq \int_M\tilde \theta_\xi e^h\omega^n=H_0(\xi)$, as shown in Proposition \ref{P-invariant}. 

If $\omega$ is not $G$-invariant, the argument above does not apply. However given the geodesic line $\rho(t)$ generated by $Y=-J\xi-\i \xi$ through an invariant metric $\omega_0$, we construct the unique geodesic segment $\phi(t), t\in [0, T]$ such that $\phi(0)=\phi, \phi(T)=\rho(T)$ for any fixed $T$. We want to understand the asymptotic behavior when $T$ is large. Roughly speaking, such geodesic segments  converge to a geodesic ray starting at $\phi$ which is \emph{parallel} to $\rho(t)$ when $T\rightarrow \infty$. This intuitive statement can be made precise and strict as in Chen \cite{Chen00}. We will only consider those geodesic segments here since it is sufficient for our purpose. 

Along the geodesic segment $\phi(t)$ connecting $\phi$ and $\rho(T)$, we consider the function
\[
A_\phi(t)=\int_M \dot \phi e^{h(t)}\omega_t^n.
\]
An essential point for our argument is that $A_\phi(t)$ is monotone.  First we note that $A_\phi(t)$ is related to Ding's $\cF$-functional defined in \cite{Ding}. 
Fix a background metric $\omega_0$ and we write $\omega=\omega_0+\i\p\bar\p \phi$. Recall Ding's $\cF$-functional is defined by
\begin{equation}
\cF_{\omega_0}(\omega)=\cE_0(\phi)+\cF_0(\phi), 
\end{equation}
where $\cE_0(\phi)$ is the Aubin-Yau functional with the form \[
\cE_0(\phi)=-\frac{1}{(n+1)}\sum_{j=1}^n\int_M \phi \omega_0^j\wedge\omega_\phi^{n-j}. 
\] 
and $\cF_0(\phi)$ takes the form
\[
\cF_0(\phi)=-V\log \left(\int_M e^{h_0-\phi}\omega_0^n\right).
\]
The Aubin-Yau functional $\cE_0(\phi)$ can be characterized by its derivative along any ($C^1$ for example) path $\phi(t)$
\[
\frac{d\cE_0(\phi(t))}{dt}=-\int_M \dot\phi \omega_\phi^n.
\]
We can also compute the derivative of $\cF_0(\phi)$ as follows,
\begin{equation}
\frac{d}{dt}\cF_0(\phi)=V\left(\int_M e^{h_0-\phi}\omega_0^n\right)^{-1}\int_M \dot \phi e^{h_0-\phi}\omega_0^n.
\end{equation}
By \eqref{E-measure}, we obtain
\begin{equation}\label{E-derivative}
\frac{d}{dt}\cF_0(\phi)=\int_M \dot \phi e^{h_\phi}\omega^n_\phi
\end{equation}
and it agrees with $A_\phi(t)$ we have defined above. We can compute directly that (along a smooth curve $\phi(t)$), as in \eqref{E-Hconvex} below, 
\begin{equation}\label{E-2nd}
\frac{d^2}{dt^2}\cF_0(\phi)=\int_M (\ddot \phi-|\nabla\dot\phi|^2)e^{h_{\phi}}\omega^n_{\phi}+\int_M (|\nabla \dot\phi|^2-(\dot\phi+a(t))^2)e^{h_{\phi}}\omega^n_{\phi}. 
\end{equation}
By \eqref{Ricci-potential}, we compute
\[
\p\bar \p \dot h=\p\bar \p (-\Delta \dot \phi-\dot \phi). 
\]
Hence there exists a time-dependent constant $a(t)$ such that 
\begin{equation}\label{E-h}
\dot h=-\Delta \dot \phi-\dot \phi-a(t).
\end{equation}
By the normalization condition \eqref{E-hn}, we can then get 
\begin{equation}\label{E-a}
\int_M (\dot \phi+a(t))e^{h(t)} \omega^n_t=0. 
\end{equation}
Now we can compute directly, using \eqref{E-h} and \eqref{E-a}, that
\begin{equation}\label{E-Hconvex}
\begin{split}
\frac{d}{dt}A_\phi(t)=&\int_M \ddot \phi e^{h(t)}\omega_t^n+\int_M \dot \phi e^{h(t)} \dot h \omega^n_t+\int_M \dot \phi e^{h(t)} \Delta_t \dot \phi \omega^n_t\\
=&\int_M (\ddot \phi-\dot\phi^2-a(t)\dot\phi)e^{h(t)}\omega^n_t\\
=&\int_M (\ddot \phi-|\nabla\dot\phi|^2_t)e^{h(t)}\omega^n_t+\int_M (|\nabla \dot\phi|^2-(\dot\phi+a(t))^2)e^{h(t)}\omega^n_t. 
\end{split}
\end{equation}

By the modified Poincare inequality \eqref{E-poincare}, we know that $\cF_0(\phi)$ is convex along (smooth) geodesics. The convexity of $\cF_0(\phi)$ and $\cF$-functional is established by Berndtsson \cite{Bern06, Bern11} with only bounded Kahler potential, using his curvature formulas. We state Berndtsson's results for convenience (for more details including notations, we refer to \cite{Bern11}). 
Let $X$ be a projective manifold with semi-negative canonical line bundle ($-K_X\geq 0$) of dimension $n$ and let $U$ be a domain in $\C$. We use $w$ to denote the coordinate in $U$.  

\begin{thm}[Berndtsson] \label{Bern}Assume that $-K_X\geq 0$ and let $\phi_w$ be a curve of metrics on $-K_X$ such that 
\[
\i \p\bar\p_{w, X} \phi_w\geq 0
\]
in the sense of current on $X\times U$, Then
\[
F(w):=-
\log\int_X e^{-\phi_w}
\]
is subharmonic in $U$. If $\phi_w$ depends only on $t$,  the real part of $w$ ($w=t+is$), then $F$ is convex in $t$. Moreover, assume $H^{0, 1}(X)=0$ and $\phi_t$ is uniformly bounded in the sense that there is a smooth metric $\psi$ on $-K_X$ such that $|\psi-\phi_t|\leq C$. Then if $F(t)$ is a linear function of $t$ in the neighborhood of $0\in U$, then there exists a holomorphic vector field (possibly $t$-dependent) $V$ on $X$ with flow $\sigma_t$ such that
\[
\sigma_t^* (\p\bar\p \phi_t)=\p\bar \p \phi_0. 
\]
\end{thm}

Given the interpretation of $A_\phi$ as the derivative of $\cF_0$, one can give a direct proof of the following fact.

\begin{lem}We have the following estimate
\begin{equation}\label{E-aineq}
A_\phi(0)=\int_M \dot\phi(0)e^h\omega^n \leq -H_0(\xi)+O(T^{-1}). 
\end{equation}
\end{lem}
\begin{proof}We know that $A_\phi(t)=\frac{d}{dt}\cF_0$. By the convexity of $\cF_0$, we know that
\[
A_\phi(0)\leq \frac{\cF_0(\phi(T))-\cF_0(\phi_0)}{T}
\]
Since $\phi(T)=\rho(T)$ and we know that $A_\rho(t)=-H_0(\xi)$ for any $t$ (equivalently, $\cF_0(\rho(t))$ is an affine function), we then have
\[
A_\phi(0)\leq \frac{\cF_0(\rho(T))-\cF_0(\rho_0)}{T}+O(T^{-1})=-H_0(\xi)+O(T^{-1}). 
\]
\end{proof}

\begin{rmk}This argument was pointed out by the referee which simplifies our original arguments. \end{rmk}

Hence by \eqref{E-h2} and \eqref{E-aineq} we obtain,
\begin{equation}
H(\omega)\geq H_0(\xi)+V-\int_M e^{-\dot \phi(0)} \omega^n+O(T^{-1}). 
\end{equation}
By Proposition \ref{P-norm}, $\int_M e^{-\dot \phi}\omega^n_t$ is constant. Hence we obtain, 
 \begin{equation}\label{E-h3}
 H(\omega)\geq H_0(\xi)+V-\int_M e^{-\dot \phi(T)}\omega^n_{\phi(T)}
 \end{equation}
When $T$ is large enough, we want to show $\dot \phi(T)\approx \dot \rho(T)$ in an effective way. To be more precise,  
we will show in Lemma \ref{L-comparison} that
\begin{equation}\label{E-comparison}
\int_M (\dot\phi(T)-\dot\rho(T))^2 \omega^n_{\rho(T)}=O(T^{-2})
\end{equation}
With this approximation, we will establish the following approximation in Lemma \ref{L-h6} below, 
\begin{equation}\label{E-h6}
\left|\int_M e^{-\dot \rho(T)} \omega^n_{\rho(T)}-\int_M e^{-\dot \phi(T)}\omega^n_{\phi(T)}\right|=O(T^{-1}).
\end{equation}
We then obtain, by \eqref{E-h3} and \eqref{E-h6}, 
\begin{equation}\label{E-h4}
H(\omega)\geq  H_0(\xi)+V-\int_M e^{-\dot \rho(T)}\omega^n_{\rho(T)}+O(T^{-1})
\end{equation}
By the definition of $\rho(t)$, in particular $\eqref{E-geodesic1}$, we have,
\[
\int_M e^{-\dot \rho(T)}\omega^n_{\rho(T)}=V.
\]
By letting $T\rightarrow \infty$, \eqref{E-h4} then implies 
\[
H(\omega)\geq H_0(\xi).
\]
This completes the proof. 
\end{proof}

Next we establish \eqref{E-comparison} and \eqref{E-h6} in the proof of Theorem \ref{T-noninvariant}. 
First we follow the comparison geometric argument in Chen \cite{Chen09} to establish \eqref{E-comparison}. The argument relies on the results established in Chen \cite{Chen00} and Calabi-Chen \cite{CalabiChen} asserting that the space of Kahler potentials with Mabuchi metric is actually an Alexanderov space of nonpositive curvature.

\begin{lem}\label{L-comparison}We have the estimate \eqref{E-comparison}.
\end{lem}

\begin{proof}

Consider the triangle in $\cH$ with vertices $B=0, C=\phi_0, D=\phi(T)=\rho(T)$. Denote the distance $|BC|=d$ and along geodesic segment $\phi(t)$, $\int_M \dot \phi^2\omega_\phi$ is constant for any $t\in [0, T]$ and we assume both are nonzero for ${BD}, {CD}$.
Then we have
\[
|BD|^2=T^2 \int_M \dot \rho^2 \omega^n_\rho, |CD|^2=T^2\int_M \dot \phi^2 \omega_\phi^n.
\]
Let $\tilde B, \tilde C, \tilde D$ be the vertices of an Euclidean triangle with the same length as $BCD$ and denote the angle at $\tilde D$ by $\tilde \theta$. Then 
\[
\cos \tilde \theta=(|BD|^2+|CD|^2-d^2)/2|BD||CD|. 
\]
The tangent vector at $D$ for $BD$ is given by $\dot \rho(T)$, for $CD$ by $\dot \phi(T)$. The inner product of two vectors is then
\[
(\dot \phi(T), \dot \rho(T))=\int_M \dot \phi \dot \rho \omega^n_{\phi(T)}. 
\]
Hence the angle $\theta$ formed at $D$ of the triangle $BCD$ is given by
\[
\cos \theta=\frac{\int_M \dot \phi \dot \rho \omega^n_{\phi(T)}}{\left(\int_M \dot \phi^2 \omega_{\phi(T)}^n\int_M \dot \rho^2 \omega_{\phi(T)}^n\right)^{1/2}}
\]
Since $\cH$ has nonpositive curvature, we know $\theta\leq \tilde \theta$, hence $0<\cos \tilde \theta\leq \cos \theta\leq 1$ (note that we consider $T$ really large, hence $\tilde \theta$ is small). It follows that
\begin{equation}\label{E-triangle0}
T^2\int_M (\dot \phi(T)-\dot \rho(T))^2 \omega^n_{\phi(T)}\leq d^2. 
\end{equation}
This proves \eqref{E-comparison}.
\end{proof}

\begin{lem}\label{L-h6}The approximation \eqref{E-h6} holds.
\end{lem}

\begin{proof}First we apply Proposition \ref{P-convexity} to get
\begin{equation}\label{E-h8}
\int_M e^{-\dot \phi}(\dot \phi-\dot \rho)\omega^n_{\rho}\leq \int_M (e^{-\dot \rho}-e^{-\dot \phi})\omega^n_{\rho}\leq \int_M e^{-\dot \rho} (\dot\phi-\dot \rho) \omega^n_{\rho}.
\end{equation}
We omit the subscript $T$ in above for simplicity. By H\"older inequality, we have
\[
\left|\int_M e^{-\dot \rho} (\dot\phi-\dot \rho) \omega^n_{\rho}\right|^2\leq\int_M e^{-2\dot \rho}\omega^n_\rho \int_M (\dot \phi-\dot \rho)^2\omega^n_\rho.  
\]
Observe that $\dot \rho=-\sigma^{*}_t\theta_\xi=-\theta_\xi \circ \sigma_t$ and hence it is uniformly bounded. It then follows from Lemma \ref{L-comparison} that 
\[
\left|\int_M e^{-\dot \rho} (\dot\phi-\dot \rho) \omega^n_{\rho}\right|\leq CT^{-1},
\]
where $C$ depends only on $\omega_0, \xi$. Similarly we have
\begin{equation}\label{E-h7}
\left|\int_M e^{-\dot \phi} (\dot\phi-\dot \rho) \omega^n_{\rho}\right|^2\leq\int_M e^{-2\dot \phi}\omega^n_\rho \int_M (\dot \phi-\dot \rho)^2\omega^n_\rho.  
\end{equation}
Along the geodesic segment we know that $\ddot \phi$ is nonnegative ($\phi$ is convex in $t$), hence 
\[
\dot \phi(T)\geq \left(\phi(T)-\phi(0)\right) T^{-1}=\rho(T)T^{-1}-\phi_0 T^{-1}
\]
It follows that 
\[
-\dot \phi(T)\leq -\rho(T) T^{-1}+\phi_0 T^{-1}
\]
Since $\dot \rho=-\theta_\xi \circ \sigma_t$ is uniformly bounded, we have that
$|\rho(t)|\leq Ct$ for any $t$.  It follows that $-\dot \phi(T)\leq C$ for some uniform constant $C$ (we assume $T\geq 1$), depending only on $\phi_0$ and $\theta_\xi$. Hence by \eqref{E-h7} and \eqref{E-comparison}, we have
\[
\left|\int_M e^{-\dot \phi} (\dot\phi-\dot \rho) \omega^n_{\rho}\right|\leq CT^{-1}. 
\]
This completes the proof of \eqref{E-h6}. 
\end{proof}

We  include a proof of the monotonicity of $A_\phi(t)$ based on the modified Poincare inequality  (see \eqref{E-Hconvex}) along $C^{1, 1}$ geodesics. We hope this alternative proof is interesting in its own right. 

\begin{lem}\label{L-mono}The function $A_\phi(t)$ is monotone increasing along a geodesic segment.
\end{lem}

\begin{proof}
By the observation in \eqref{E-measure} above $e^h\omega^n$ is well-defined for any bounded generalized Kahler potential.  First suppose $\phi(t)$ is a smooth geodesic path of Kahler potentials,  then  by \eqref{E-derivative} and \eqref{E-2nd}, $A_\phi(t)$ is monotone increasing by Proposition \ref{Futaki} (\eqref{E-poincare}).
Now suppose $\phi_0, \phi_T$ are two smooth Kahler potentials in $[\omega_0]$ and we consider the $C^{1, 1}$ geodesic segment $\phi(t)$ connecting the given two Kahler potentials.  We use an approximation argument. For any $\epsilon>0$, there exists an approximating smooth geodesic $\phi_\epsilon(t)$ and we can associate such a path
\begin{equation}\label{E-Afunction}
A_\epsilon(t)=\int_M \dot \phi_\epsilon(t) e^{h_\epsilon(t)}\omega^n_\epsilon(t)=\l_\epsilon(t)\int_M \dot \phi_\epsilon(t) e^{h_0-\phi_{\epsilon}(t)}\omega^n_0,\end{equation}
where $\l_\epsilon(t)$ is a time dependent constant such that 
\[
\l_\epsilon(t)\int_M  e^{h_0-\phi_{\epsilon}(t)}\omega^n_0=V.
\]
A direct computation as above shows that
\[
\frac{dA_\epsilon }{dt}=\int_M (\ddot \phi_\epsilon-|\nabla \dot\phi_\epsilon|^2) e^{h_\epsilon}\omega_\epsilon^n+\int_M \left(|\nabla \dot \phi_\epsilon|^2-(\dot \phi_\epsilon+a_\epsilon(t))^2\right)e^{h_\epsilon}\omega_\epsilon^n>0.
\]
It then follows that
\[
A_\epsilon(0)<A_\epsilon(T). 
\]
Since $\phi_\epsilon(t)\rightarrow \phi(t)$ in $C^{1, \alpha}$ when $\epsilon\rightarrow 0$, we have
\[
A_\epsilon(t)=A_\phi(t)+o(\epsilon).
\]
The desired monotonicity follows by letting $\epsilon\rightarrow 0$. 
\end{proof}

In general, one can associate an invariant for any geodesic rays and obtain a lower bound of $H(\omega)$ in terms of such an invariant with some extra efforts. Suppose $\rho(t)$ is a geodesic ray with $\rho(t)\in \cH_\infty$. By an observation of Berndtsson \cite{Bern11},  $\rho(t)$ is Lipschitz in $t$-direction, and hence $\dot\rho$ is a $L^\infty$ function in $t$. 
\begin{defi}Let $\rho(t)$ be a geodesic ray with $\rho(t)\in \cH_\infty$, define the function
\[
Y_\rho(t)=-\int_M \dot \rho e^{h_\rho}\omega^n_\rho-V\log\left(V^{-1}\int_M e^{-\dot\rho}\omega^n_\rho\right),
\]
where we understand the notion as follows, 
\[
e^{h_\rho}\omega^n_\rho=\l e^{h_0-\rho}\omega_0^n, \;\mbox{with}\; \l=V\left(\log \int_M e^{h_0-\phi}\omega^n_0\right)^{-1}.
\]
\end{defi}
By the discussion above, we know that
\[
\frac{d\cF_0(\rho(t))}{dt}=\int_M \dot \rho e^{h_\rho}\omega^n_\rho
\]
and hence it is increasing by Berndtsson's result. Berndtsson's result (Proposition \ref{P-norm}) holds in a more general setting. Hence 
 $Y_\rho(t)$ is a decreasing function and we can define
\[
Y_\rho=\lim_{t\rightarrow \infty} Y_\rho(t). 
\]
Using the very similar ideas as above, one can obtain that, for any geodesic ray $\rho$ (with some mild assumption, say $|\rho(t)|\leq Ct$ and $\rho(t)$ is $C^1$), \[H(\omega)\geq Y_\rho.\]
We shall not pursue this generality here. 
\subsection{Modified $\cF$-functional}

When the Futaki invariant is not zero, or equivalently, the extremal holomorphic vector field $X$ is not zero, Tian-Zhu \cite{Tianzhu00} introduced the notion of modified $\cF_X$ functional. This functional is important for the study of Kahler-Ricci solitons. 
Recall that the $\cF_X$ functional is defined by, given any path $\phi(t)$ connecting $\omega$ and $\omega_\phi=\omega+\p\bar\p \phi$, 
\begin{equation}
\cF_X(\phi)=-\int_0^1 \int_M \dot \phi_t e^{\theta_X(\phi_t)} \omega_{\phi_t}^n dt-V\log \left(\frac{1}{V}\int_M e^{h_g-\phi} \omega^n\right), 
\end{equation}
where $h$ is the Ricci potential of $\omega$, and $\theta_X(\omega)$ is the potential of $X$ with respect to $\omega$ ($\iota_X\omega=\i \bar \p \theta_X$), both satisfying the normalization,
\[
\int_M e^h \omega^n=\int_M e^{\theta_X} \omega^n=\int_M \omega^n=V.
\]
Tian-Zhu proved that $\cF_X$ is independent of the path and 
its first  variation formula of $\cF_X$ is given by, 
\begin{equation}\label{E-Fvariation1}
\frac{d\cF_X}{dt}=- \int_M \dot \phi e^{\theta_X(\phi)}\omega_\phi^n+V\left(\int_M e^{h-\phi}\omega^n\right)^{-1}\int_M \dot \phi e^{h-\phi}\omega^n.
\end{equation}
The Euler-Lagragian equation is given by
\[
e^{\theta_X(\phi)}\omega^n_\phi=Ve^{h-\phi}\omega^n \left(\int_M e^{h-\phi}\omega^n \right)^{-1},
\]
which is exactly the equation for a Kahler-Ricci soliton with extremal vector field $X$.
By \eqref{E-measure}, we can write \eqref{E-Fvariation1} as
\begin{equation}\label{E-F2}
\frac{d\cF_X}{dt}=- \int_M \dot \phi e^{\theta_X(\phi)}\omega_\phi^n+\int_M \dot\phi e^{h_\phi}\omega^n_\phi. 
\end{equation}

We can then compute the  second derivative of $\cF_X$, 
\begin{prop}The second variation of $\cF_X$ is given by
\begin{equation}\label{E-Fvariation2}
\begin{split}
\frac{d^2 \cF_X}{dt^2}=&-\int_M (\ddot \phi-|\nabla \dot \phi|_\phi^2)\left(e^{\theta_X(\phi)}\omega_\phi^n-Ve^{h-\phi}\omega^n\right)\\
&+\int_M (|\nabla \dot \phi|^2_\phi-|\dot \phi+a(t)|^2) e^{h_\phi}\omega^n_\phi.
\end{split}
\end{equation}
\end{prop}

\begin{proof}We compute
\[
\begin{split}
\frac{d}{dt}\int_M \dot \phi e^{\theta_X(\phi)}\omega_\phi^n=&\int_M (\ddot \phi+\dot \phi X(\dot \phi)+\dot \phi\Delta \dot \phi)e^{\theta_X(\phi)}\omega^n_\phi\\
=&\int_M (\ddot\phi-|\nabla \dot \phi|^2) e^{\theta_X(\phi)}\omega^n_\phi,
\end{split}
\]
where we use integration by parts and the fact that $\nabla_\phi \theta_X(\phi)=X$. The other part is exactly the same as the computation of $\cF$-functional. 
\end{proof}

Invoking Berndtsson \cite{Bern11}, we have the following, 
\begin{prop}
The $\cF_X$ functional is  convex along any $C^{1, 1}$ geodesic $\phi(t)\in \cH_{1, 1}$. If $\cF_X(\phi(t))$ is a linear function in $t$, then there exists a one-parameter holomorphism $\sigma_t: M\rightarrow M$ such that \[\sigma_{\phi(t)}=\sigma_t^{*}\omega_{\phi_0}.\]
\end{prop}

\begin{proof}The only thing we need in addition is that $\theta_X(\phi)$ is uniformly bounded. This follows from \cite{Zhu00} and the proof can be extended directly to $C^{1, 1}$ potentials. The statement then follows Berndtsson \cite{Bern11}. 
\end{proof}

\begin{rmk}
As a consequence of this convexity one can give a proof of Tian-Zhu's  result \cite{Tianzhu00} on the uniqueness of Kahler-Ricci soliton in $(M, [\omega_0])$. The only additional fact we need is that the extremal vector field for Kahler-Ricci soliton is unique up to automorphisms \cite{Tianzhu02}.  In Kahler-Einstein case, Berman \cite{Berman} and Berndtsson \cite{Bern11} have already given a new proof of Bando-Mabuchi's uniqueness theorem \cite{BM} using such convexity directly. We learned from Song Sun that Berndtsson can extend his results in \cite{Bern11} to give a new proof of Tian-Zhu's uniqueness theorem. 
\end{rmk}

\subsection{Perelman's $\mu$-functional}

As a direct consequence of Theorem \ref{T-noninvariant}, we can obtain an upper bound for Perelman's $\mu$-functional  and Kahler Ricci soliton, if exists, maximizes the $\mu$ functional. Recall Perelman's W-functional \cite{Perelman01} is defined as
\[
W(g, \tau, f)=\frac{1}{(4\pi \tau)^{-n/2}}\int_M (\tau (R+|\nabla f|^2)-n+f)e^{-f}dv_g,
\]
where $R$ is the scalar curvature of the metric $g$. 
On Fano manifolds, however, it is more convenient to let $\tau=1/2$ and it is also convenient using (complex) geometric quantities which differ only by multiple of a constant.
Hence we consider $W$-functional on $(M, [\omega_0])$ by
\begin{equation}\label{E-w}
W(\omega, f)=\int_M (R+|\nabla f|^2+f)e^{-f}\omega^n,
\end{equation}
with the normalization condition
\[
\int_M e^{-f}\omega^n=\int_M \omega^n=V. 
\]
And the $\mu$-functional is defined to be
\[
\mu(\omega)=\inf_{f}W(\omega, f). 
\]
By Rothaus \cite{Ro}, there always exists some smooth function $f$ to minimize $W(\omega, \cdot)$ and such a minimizer $f$ satisfies the equation
\begin{equation}\label{E-gs}
2\Delta f+f+R-|\nabla f|^2=V^{-1}\mu. 
\end{equation}
When $\omega_0$ is a Kahler-Ricci soliton, then $f=-h_0$. An easy way to see this fact is to use a result of Sun-Wang \cite{SW} (Lemma 3.4), where they prove that there exists a unique  solution of \eqref{E-gs} if $g$ is a gradient shrinking Ricci-soliton. Suppose $\omega_0$ is a Kahler-Ricci soliton. Then 
\[
\Delta h_0+|\nabla h_0|^2+h_0=constant. 
\]
We can write the scalar curvature as $R=n+\Delta h_0$. 
It follows that $-h_0+c$ is a solution of \eqref{E-gs} for an appropriate constant $c$. 
By the uniqueness and the normalization condition, we know $f=-h_0$. In particular,
\[
\mu(\omega_0)=W(\omega_0, f)=W(\omega, -h_0). 
\]

\begin{cor}\label{cor-mu}For any $\omega\in [\omega_0]$, 
\[
\mu (\omega)+H(\omega)\leq nV.
\]
Hence Perelman's $\mu$-functional is bounded above by, 
\[
\mu(\omega)\leq nV-N_X. 
\]
In particular, if $\omega_0$ is a Kahler-Ricci soliton, then for any $\omega\in [\omega_0]$, 
\[
\mu(\omega)\leq \mu(\omega_0)=nV-N_X. 
\]
\end{cor}

\begin{proof}
We  observe that, since $R=n+\Delta h$, 
\[
\mu(\omega)\leq W(\omega, -h)=\int_M \left((n+\Delta h)+|\nabla h|^2-h\right)e^h\omega^n= nV-H(\omega). 
\]
When $\omega_0$ is a Kahler-Ricci soliton, then $h_0$ is the Hamiltonian of $\xi_0$ (the imaginary part of the extremal vector field),  
\[
H(\omega_0)=\int_M h_0 e^{h_0} \omega_0^n=\int_M \theta_{\xi_0}e^{h_0}\omega_0^n=H_0(\xi_0)=N_X.
\]
By the discussion above, we know
\[
\mu(\omega_0)=W(\omega, -h_0)=nV-N_X. 
\]
\end{proof}

For metrics which is invariant with respect to $Im(X)$, such an upper bound was proved in a recent paper by Tian-Zhang-Zhang-Zhu \cite{TZ3} and the authors proved that 
the Kahler-Ricci flow with any invariant  initial metric actually maximizes $\mu$ functional asymptotically 
provided that the modified Mabuchi functional is bounded below, and it is used to give an alternative proof of convergence of the Kahler-Ricci flow for invariant metric assuming the existence of Kahler-Ricci soliton. It is a belief that if Kahler-Ricci soliton exists, then the Kahler-Ricci flow would actually converge to the soliton, without invariant assumption on initial metric. In \cite{TZ3} the authors proposed a conjecture (Conjecture 3.3), which suggests how one might prove such a convergent result.  Corollary \ref{cor-mu} seems to strengthen this belief.

\section{Kahler-Ricci soliton and geodesic stability}

First we recall the following Donaldson's conjecture \cite{Donaldson99}, 
\begin{conj}[Donaldson] The following are equivalent:
\begin{enumerate}
\item There is no constant scalar metric in $(M, [\omega_0])$.
\item There is a geodesic ray $\phi(t), t\in [0, \infty)$ such that the derivative of Mabuchi's K-energy, for all $t\in [0, \infty)$, 
\[
\frac{d\cK}{dt}=-\int_M \dot\phi (R-\underline {R})\omega^n_\phi<0.
\]

\item For any Kahler potential $\phi$, there exists a geodesic ray as in (2) starting at $\phi$. 
\end{enumerate}
\end{conj}

The conjecture remains open in both directions (see \cite{Chen08} for partial results). One technical difficulty is that  the geodesic rays  do not have enough regularity to talk about the derivative of $\cK$-energy in general. But $\cF$-functional and its derivative make sense even for geodesics with only bounded potential. When $(M, [\omega_0])$ is Fano, it is very natural to  formulate the conjecture in terms of $\cF$-functional.

\begin{conj}
\label{C-stability}
Let $(M, [\omega_0])$ be a Fano manifold. Then the following are equivalent

\begin{enumerate}
\item There exists  no  Kahler-Einstein metric in $(M, [\omega_0])$.

\item There exists a geodesic ray $\phi(t)$ such that for all $t\in [0, \infty)$
$\frac{d\cF }{dt}<0. $

\item For every point $\phi\in \cH$, there exists a geodesic ray as in (2) starting at $\phi$. 
\end{enumerate}
\end{conj}

From analytic point of view, Kahler-Ricci soliton is a natural generalization of Kahler-Einstein metric.  We extend
the discussion for Kahler-Einstein metrics to Kahler-Ricci solitons. We formulate a version of geodesic stability for Kahler-Ricci soliton, with the hope that it will motivate an algebro-geometric notion of stability for Kahler-Ricci solitons. The key notion in our formulation is the modified $\cF$ functional introduced by Tian-Zhu \cite{Tianzhu00}.
\begin{conj} Let $(M, [\omega_0])$ be a Fano manifold. Recall $X$, the extremal holomorphic vector field, is determined \emph{a priori}. We consider $Im(X)$ invariant metrics in $\cH$.  The following are equivalent,

\begin{enumerate}
\item There exists  no  Kahler-Ricci soliton in $(M, [\omega_0])$.

\item There exists a geodesic ray $\phi(t)$ such that for all $t\in [0, \infty)$
$\frac{d\cF_X }{dt}<0. $

\item For every point $\phi\in \cH$, there exists a geodesic ray as in (2) starting at $\phi$. 
\end{enumerate}
\end{conj}

\end{document}